\documentclass[11pt]{article}
\usepackage{amsthm, amssymb}
\usepackage{graphicx}
\usepackage{amsmath,amsthm, amsxtra,amssymb,latexsym, amscd}
\usepackage[latin1]{inputenc}
\theoremstyle{plain}
\newtheorem{theorem}{Theorem}[section]

\newtheorem{corollary}[theorem]{Corollary}

\theoremstyle{definition}
\newtheorem{definition}[theorem]{Definition}
\newtheorem{example}[theorem]{Example}
\newtheorem{algorithm}[theorem]{Algorithm}

\theoremstyle{remark}
\newtheorem{remark}[theorem]{Remark}

\def\bn{{\mathbb N}}
\def\bz{{\mathbb Z}}

\def\Acal{{\mathcal A}}

\def\mod{\mathop{\rm mod}\nolimits}

\def\max#1{{\rm max}\,\{#1\}}

\def\ini#1{{\rm in}\,(#1)}

\def\card{{\rm card}}

\def\gbb{Gr\"obner basis\ }

\def\ul{\underline }

\title{Gr\"obner basis, a "pseudo-polynomial" algorithm for computing the  Frobenius number}
\author{
Marcel Morales\\
{\small 
Institut Fourier, Laboratoire de Math\'ematiques
associ\'e au CNRS, UMR 5582,}\\
{\small Universit\'e Joseph Fourier, B.P.74,
38402 Saint-Martin d'H\`eres cedex, France}\\ 
{\small and ESPE Université de Lyon,  France}\\
{\small{\tt marcel.morales@ujf-grenoble.fr}}\\
\and  Nguyen Thi Dung\\
{\small Thai Nguyen University of Agriculture and Forestry, 
 }\\
{\small 
Thai Nguyen,  Vietnam }\\
}

\begin{document}

\hfuzz=3pt
\maketitle
%%\date{}
%\renewcommand{\abstractname}{Abstract (in english)}
{\small \sc Abstract.}
%%{\footnote {  \date{}}}
 
 Let consider $n$ natural numbers    $a_1  ,\ldots ,  a_{n}  $. Let $S$ be the numerical semigroup generated by $a_1  ,\ldots ,  a_{n}  $. 
Set $A=K[t^{a_1}, \ldots , t^{a_n}]=K[{x_1}, \ldots , {x_n}]/I$.
 The aim of this paper is:  
\begin{enumerate}
\item Give an effective  pseudo-polynomial algorithm on $a_1$,  which computes The Apéry set and the Frobenius number of $S$.
 As a consequence it also solves in pseudo-polynomial time the  integer knapsack problem : given a natural integer b, b belongs to $S$?
\item The \gbb of $I$ for the reverse lexicographic order to $x_n,\ldots ,x_1$, without using Buchberger's algorithm. 
\item $\ini{I} $ for the reverse lexicographic order to $x_n,\ldots ,x_1$.
\item $A$ as a $K[t^{  a_1 }]$-module.

 \end{enumerate}
 We dont know the complexity of   our algorithm.  We need to solve the "multiplicative" integer knapsack problem:
 
Find all positive integer solutions $({k_1}, \ldots , {k_n})$ of the inequality $\prod_{i=2}^n (k_i+1)\leq a_1+1$.

 This algorithm is easily implemented.
 The implementation of this algorithm "frobenius-number-mm", for $n=17 $, can be downloaded in \hfill\break
https://www-fourier.ujf-grenoble.fr/~morales/frobenius-number-mm
\vskip1truecm\noindent
\section*{Introduction} 
In the sequel we shall use the following notations.
Let $K$ be a field, $\Acal$ be a set of $n$ natural numbers  $\Acal =\{{  a_1 },\ldots ,{  a_{n} }\}\subset \bn$.  
  $S$   the numerical semigroup generated by $a_1,\ldots ,a_n$, that is $S=\{k_1 a_1+\ldots k_n a_n | k_i\in \bn\}$. We consider the  one-dimensional {\it toric affine ring} $A=K[t^{a_1}, \ldots , t^{a_n}]\subset K[t]$,
 that is  $A= K[t^k | k\in S]:=K\lbrack S\rbrack$.\hfill\break 
The ring  $A=K[t^{a_1}, \ldots , t^{a_n}]\subset K[t]$ has a presentation as a quotient of the polynomial ring $K[x_1,\ldots ,x_n]$, as follows:\hfill\break 
Let $\varphi : K[x_1,\ldots ,x_n]\rightarrow K[t^{a_1}, \ldots , t^{a_n}]$ defined by 
\begin{eqnarray*}
x_1&     \mapsto    & t^{a_1}\\
&\vdots&\\
x_n&     \mapsto    & t^{a_n}\\
\end{eqnarray*}
Let $ I(S)     $ be the kernel of $\varphi $, that is the ideal 
ideal
formed by all  polynomials  of $
K[x_{1},\ldots,x_{n}]$ such that  $P(t^{a_1}, \ldots , t^{a_n})=0$.

The ideal $I(S)$ has a system of generators formed by binomials which
are
differences of two monomials with coefficient 1. Note that if we grade the polynomial ring $K[x_1,\ldots ,x_n]$ by setting $\deg(x_i)=a_i$, the
morphism $\varphi$ is homogeneous, and the ideal $I(S)$ is homogeneous.
  
  The following theorem is well known, we give here a short proof for the commodity of the reader. 
\begin{theorem} Suppose that $a_1,\ldots ,a_n$ are relatively prime numbers then any large integer belongs to $S$.
\end{theorem}
\begin{proof} Suppose that $n=2$, By Bézout'theorem there exist relative integer numbers $s_1,s_2$ such that $s_1a_1+s_2a_2=1$. We can assume 
 that $s_1>0 , s_2<0$. Let $k>0$ big enough we can write $k=qa_2+r$ with $0\leq r<a_2$, which implies $k=qa_2+r(s_1a_1+s_2a_2)=rs_1a_1+(q+rs_2)a_2$.
 Since $k$ is
 large enough $(q+rs_2)>0$, hence $k\in S$. 

A similar argument works for $n>2$. 
\end{proof}
\begin{definition} Suppose that $a_1,\ldots ,a_n$ are relatively prime numbers, the biggest integer number in $\bn\setminus  S$ 
is called the Frobenius number, we denote it by $g(S).$ More generally if $\gcd(a_1,\ldots ,a_n)=\lambda $ then the biggest integer in $\lambda \bn\setminus  S$ 
is called the Frobenius number, we denote it by $g(S).$
\end{definition}
Suppose that   $\gcd(a_1,\ldots ,a_n)=\lambda $, let $\widetilde S$  be the semigroup generated by $\frac{a_1}{\lambda },\ldots ,\frac{a_n}{\lambda }$. 
We have that $g(S)=\lambda g(\widetilde S).$
The problem of computing the Frobenius number is open since the end of 19th Century, for $n=2$ there is a formula (see section 1), for $n=3$ a formula using
Euclide's algorithm for gcd was given in \cite{rod}. In 1987, in \cite{Mo1} the first author translate for the first time the Frobenius problem into an algebraic setting,
showing that the Frobenius number  is the degree of the Hilbert-Poincaré series written as a rational fraction, moreover By using \cite{rod}, 
in the case $n=3$ the Hilbert-Poincaré series is completely described by an   algorithm using only Euclide's 
algorithm for gcd, that is of complexity $\ln(a)$. An implementation in Pascal was done 
by the first author to compute a system of generators of the affine  monomial curve $K[t^{a}, t^{b}, t^{c}]$ and its projective closure,
 which computed the Frobenius number for three natural numbers.
 In recent works \cite{ein} and \cite{rou}, the  computation of Frobenius number, is related to the computation of the Hilbert-Poincaré series. More precisely,
 in
 \cite{rou} the author deduces the Frobenius number from a \gbb of the ideal $I(S)$. 
We recall that the computation of a \gbb is double exponential complexity by using the Buchberger algorithm.

In this paper we study the Frobenius problem from algebraic point of view, this allows us to give a conceptual frame to our algorithm. 
As a consequence of our study we get a very fast and simple algorithm to compute the Frobenius number, determine completely the semigroup $S$
 and solve the knapsack integer problem, that is to decide if an integer number belong to $S$.
We have implemented for $n\leq 17$.
We need to solve the "multiplicative" integer knapsack problem:
 
Find all positive integer solutions $({k_1}, \ldots , {k_n})$ of the inequality $\prod_{i=2}^n (k_i+1)\leq a_1+1$.

 Moreover let $G(S)$ be a \gbb for the degree reverse lexicographic order to $x_n,\ldots ,x_1$.
$\prec _{revlex}$ and   $in(I(S))$ be its initial ideal. 
As an extension we find several algorithms :
\begin{enumerate}
\item The set of monomials $\{x_2^{k_2}\ldots x_n^{k_n}\notin in(I(S))\}$.
\item The system of monomial generators of $\ini{I} $. 
\item The \gbb $G(S)$ of $I$, without using Buchberger's algorithm. 
 \end{enumerate}

They are extensions of the previous work and algorithm by the first author in  \cite{Mo1}, \cite{Mo2}. 
The algorithm presented here  is implemented and can be downloaded in
https://www-fourier.ujf-grenoble.fr/~morales/. Note that because of the limitation of the Compiler for the moment the software works only for numbers less
than 10000, but an implementation in Mathematica should allow to compute with any number of digits.

In the first section we introduce  the Apéry set and we prove some known results. 

In the second section we present the connection between  Hilbert-Poincaré series and the Frobenius number. 
This connection was established by the first author for the first time in \cite{Mo1}.
 
 In the third section we introduce Noether normalization and we prove the connection between Apéry sets and Noether normalization.
 
In the last section we develop our algorithm.

In our work in progress, we will extend the above algorithm to compute \gbb of any simplicial monomial ideal. 
\section{Frobenius number, Apéry set }
For a survey on the Frobenius number we refer to \cite{ra}.
\begin{definition}Suppose that $a_1$ is the smallest among $a_1,\ldots ,a_n$. 
The Apéry set ${\rm Ap}(S,a_1)$ of the semigroup $S$ with respect to $a_1$ 
 is the set 
${\rm Ap}(S,a_1):=\{ s\in S | s-a_1 \notin S\}.$
\begin{remark} The definition of Apéry set  makes sense even if the numbers $a_1,\ldots ,a_n$ are not relatively prime numbers. 
Suppose that   $\gcd(a_1,\ldots ,a_n)=\lambda $, let $\widetilde S$  be the semigroup generated by $\frac{a_1}{\lambda },\ldots ,\frac{a_n}{\lambda }$. 
We have that ${\rm Ap}(S,a_1) $ is obtained from ${\rm Ap}(\widetilde S,\frac{a_1}{\lambda }) $ by multiplication by $\lambda $.
\end{remark} 
 \begin{theorem} (Apéry \cite{ap})Suppose that $a_1,\ldots ,a_n$ are natural numbers such that $\gcd(a_1,\ldots ,a_n)=\lambda $. 
 \begin{enumerate}
 \item ${\rm Ap}(S,a_1):=\{ \lambda w_0,\ldots ,\lambda w_{{\frac{a_1}{\lambda }-1}}\},$ where $w_i $ is the smallest element is $\widetilde S$ congruent to
$i$ mod $\frac{a_1}{\lambda }$.
\item $g(S) = \max {s-a_1 | s\in {\rm Ap}(S,a_1) }$.
\end{enumerate}
\end{theorem}
\begin{proof} \begin{enumerate}
 \item We can assume that $a_1,\ldots ,a_n$ are  relatively prime numbers. \hfill\break 
First we prove that for all $i=0,\ldots ,{a_1}-1$, $w_i$ belongs to ${\rm Ap}(S,a_1)$. Suppose that it is not true, that is $w_i-a_1\in S$ for some 
$i=0,\ldots ,{a_1}-1$.
It follows that $w_i-a_1<w_i$ and both $w_i-a_1, w_i \in S$ are   congruent to
$i$ mod $a_1$. This is a contradiction with the definition of $w_i$.
As a consequence ${\rm Ap}(S,a_1)$ has at least $a_1$ elements in order to prove the claim it will be enough to show that 
${\rm Ap}(S,a_1)$ has exactly $a_1$ elements. Suppose that $\card({\rm Ap}(S,a_1))>a_1$, then there exists two elements $s_1<s_2$ in  ${\rm Ap}(S,a_1)$
such that both $s_1<s_2$ are   congruent to
$i$ mod $a_1$ for some $i=0,\ldots ,{a_1}-1$, that is $s_2=s_1+ka_1$ with 
$k>0$ a natural integer, hence $s_2\not\in {\rm Ap}(S,a_1)$, a contradiction. 
\item  Let $h\in \bn$ such that $h>\max {s-a_1 | s\in {\rm Ap}(S,a_1) }$, since $h$ is  congruent to
$i$ mod $a_1$ for some $i=0,\ldots ,{a_1}-1$, we can write $h=w_i+\alpha a_1$, with $\alpha \in \bz$, hence
$h=(w_i-a_1)+(\alpha+1) a_1$, since $h>(w_i-a_1)$ we have $(\alpha+1)>0$, hence $ \alpha\geq 0$, which implies that $h\in S$.

\end{enumerate}
\end{proof}

\begin{corollary} Let $n=2$ suppose that  $a_1,a_2$ are relatively prime numbers then $g(S)= (a_1-1)(a_2-1)-1$.
\end{corollary}
\begin{proof}We give a combinatorial proof using Apéry sets. Since $a_1,a_2$ are relatively prime numbers, we have  
$${\rm Ap}(S,a_1):=\{ 0, a_2,\ldots , ({a_1}-1)a_2\},$$ hence $g(S)= (a_1-1)(a_2)-a_1=(a_1-1)(a_2-1)-1$

\end{proof}
For $n=2$, we will give an algebraic proof later.
\begin{corollary} For $i=0,\ldots ,{a_1}-1$ let $S_i = \{s\in S| s\equiv i \mod a_1\}$. Then $S$ is the disjoint union of $S_0,\ldots S_{a_1-1}$.
\end{corollary}
\section{Frobenius number and Hilbert-Poincaré series }
Let $R:=K[x_{1},\ldots,x_{n}]$ be the polynomial ring graded by the weights  $\deg x_1=a_1,\ldots ,\deg x_n=a_n$, and 
$I\subset K[x_{1},\ldots,x_{n}]$  be a graded ideal. Let $B=R/I$,
 the Hilbert-function of  $B$ is  defined by $H_B(l)=\dim_K B_l ,$  for all $l \in \bz$, and the Hilbert-Poincaré series of $B$:
 $$P_B(u)=\sum_{l\in \bz}H_B(l)u^l. $$ 
We recall the following Theorem from \cite{Mo1}
\begin{theorem}\label{frob-degpoincare} Let $R:=K[x_{1},\ldots,x_{n}]$ be the polynomial ring graded by the weights  $\deg x_1=a_1,\ldots ,\deg x_n=a_n$,  
$I\subset K[x_{1},\ldots,x_{n}]$  be a graded ideal and $B:=R/I$. Then 
\begin{enumerate}
 \item The Hilbert-Poincaré series of $B$ is a rational function:
$$P_B(u)=\frac{Q_B(u)}{(1-u^{a_1}) (1-u^{a_2})\ldots  (1-u^{a_n})}\ \ ,$$
where $Q_B(u)$ is a  polynomial on $u$. 
\item There exists $h$ polynomials with integer coefficients $\Phi _{H_B,0}(l),\ldots , \Phi _{H_B,h}(l) $ such that  
 $H_B(lh+i)=\Phi _{H_B,i}(l)  $   for $0\leq i\leq h-1$ and $l$ large enough.
  We recall that the index of regularity of the Hilbert function is the biggest integer $l $ such that 
 $H_B(l)\not=\Phi _{H_B,i}(l) ,$ for any $i$. 
\item The index of regularity of the Hilbert function equals the degree of the rational fraction defining the Poincaré series.
 \end{enumerate}
\end{theorem}

\begin{corollary}\cite{Mo0} Let $  S$  be the semigroup generated by $a_1,\ldots ,a_n$, and $A=K[t^{a_1}, \ldots , t^{a_n}]\subset K[t]$.
 The Frobenius number $g(S) $ coincides with 
the  degree of the rational fraction defining the Poincaré series $P_A(u)$ by the theorem \ref{frob-degpoincare}.

\begin{proof} The Hilbert function of $A$ is given by
$$H_A(l)=\begin{cases} 1 &\mbox{if } l\in S\\
0 & \mbox{if } l\notin S. \end{cases}$$
In particular if $a_1,\ldots ,a_n$ are relatively prime, $H_A(l)=1$ for $l$ large enough, and the Frobenius number coincides with 
the index of regularity of the Hilbert function $H_A(l)$, so it is the degree of the rational fraction defining the Poincaré series $P_A(u)$.
\end{proof}
\end{corollary}
\section{\gbb}
Let    $  a_1 ,\ldots ,  a_{n}  $ be natural numbers,  $\lambda =\gcd (a_1 ,\ldots ,  a_{n})$. We denote by $S$ (resp. $\widetilde S$) the semigroup generated
by $a_1 ,\ldots ,  a_{n}$ (resp. by $a_1/\lambda  ,\ldots ,  a_{n}/\lambda $). Note that the semigroup rings $K[S], K[\widetilde S]$ are isomorphic.
\hfill\break 
Let $R:=K[x_{1},\ldots,x_{n}]$ be the polynomial ring graded by the weights  $\deg x_1=a_1,\ldots ,\deg x_n=a_n$.  
We consider $\prec _{degrevlex}$  the degree reverse lexicographical order with $x_n,\ldots ,x_1$, 
but all results are valuable for any monomial order such that the variable $x_1$ never appears in a minimal system of generators of the initial ideal $in(I(S))$.
The first statement of the following theorem is an extension to the quasi-homogeneous case of \cite{LJ}.
\begin{theorem}\label{noether} Let  $A:=K[t^{a_1}, \ldots , t^{a_n}]\simeq R/I(S)$.
\begin{enumerate}
 \item The polynomial ring $K[x_{1}]\subset A$ is a Noether normalization. Moreover let $G(S)$ be a \gbb for 
$\prec _{revlex}$ and   $in(I(S))$ be the initial ideal then
 $$A\simeq \oplus_{x_2^{k_2}\ldots x_n^{k_n}\notin in(I(S))} K[t^{a_1}][t^{k_2a_2+\ldots +k_na_n }].$$ 
 \item  The Hilbert-Poincaré series is given by: $$P_A(t)=\frac{\sum_{x_2^{k_2}\ldots x_n^{k_n}\notin in(I(S))}t^{k_2a_2+\ldots +k_na_n }}{1-t^{a_1}} $$
 \item The Frobenius number $g(\widetilde S)=\frac{\max { k_2a_2+\ldots +k_na_n |  x_2^{k_2}\ldots x_n^{k_n}\notin in(I(S))}-{a_1}}{\lambda }$.
\end{enumerate}\end{theorem}
\begin{proof}
\begin{enumerate}
 \item 
We have that for any $i=2,\ldots ,n$,  $(t^{a_i})^{a_1}-(t^{a_1})^{a_i}=0$, so  $K[t^{a_1}, \ldots , t^{a_n}]$ is integral over $K[t^{a_1}]$, 
both rings have dimension one so  $K[t^{a_1}]\subset K[t^{a_1}, \ldots , t^{a_n}]$ is a Noether normalization, 
also both rings are Cohen-Macaulay. By the
 Auslander-Buschsbaum formula we get that $ K[t^{a_1}, \ldots , t^{a_n}]$ is a free $K[t^{a_1}]$-module. This is the same to say that $R/I(S)$ is a free 
 $K[x_{1}]$-module. Since $R/I(S)$ is a graded  $K[x_{1}]$-module, we can use Nakayama's lemma, hence any $K-$basis of $R/(I(S),x_1)$ gives us a basis 
 of $R/I(S)$ as a free 
 $K[x_{1}]$-module. Let $G(S)$ be a \gbb for 
$\prec _{revlex}$ and   $in(I(S))$ be the initial ideal, by definition of $\prec _{degrevlex}$, $x_1$ does not divides any of the elements in $in(I(S))$.
 On the other hand Macaulay's theorem \cite{Ei}[Theorem 15.3]says us that the set of monomials not in $ in(I(S))$ is a basis of $R/I(S)$ as a free 
 $K[x_{1}]$-module.
 
 \item It is clear  that the Hilbert-Poincaré series of $K[t^{a_1}]$ is  $\frac{1}{1-t^{a_1}} $,
   the Hilbert-Poincaré series is an additive function, hence we have the formula for the Hilbert-Poincaré series of $A$.
  \item By \ref{frob-degpoincare} The Frobenius number of $S$ is the degree of the Hilbert-Poincaré series of $A$. 
\end{enumerate}
 \end{proof}

\end{definition}
 We have the following consequence which will be important for our algorithm:
 
 \begin{corollary}\label{ap1}
We have that \begin{enumerate}
 \item ${\rm Ap}(S,a_1)=\{ k_2a_2+\ldots +k_na_n |  x_2^{k_2}\ldots x_n^{k_n}\notin in(I(S))\}$. In particular
$$ \card\{x_2^{k_2}\ldots x_n^{k_n}\notin in(I(S))\}=\frac{a_1}{\gcd(a_1,\ldots ,a_n)}.$$
  
 \item Let $s\in {\rm Ap}(S,a_1)$, such that $s=k_2a_2+\ldots +k_na_n $ and $  x_2^{k_2}\ldots x_n^{k_n}\notin in(I(S))$. Suppose that 
$s=l_2a_2+\ldots +l_na_n $ for some natural numbers $l_2,\ldots ,l_n$, then $x_2^{k_2}\ldots x_n^{k_n}\prec _{revlex} x_2^{l_2}\ldots x_n^{l_n} $.
 \end{enumerate}
 \end{corollary}
 \begin{proof}We can assume that $\gcd(a_1,\ldots ,a_n)=1$. 
\begin{enumerate}
 \item By the above theorem $K[t^{a_1}, \ldots , t^{a_n}]\simeq \oplus_{s_i\in H} K[t^{a_1}][t^{s_i }]$
 where $H=\{ k_2a_2+\ldots +k_na_n |  x_2^{k_2}\ldots x_n^{k_n}\notin in(I(S))\}$, now we prove that $H={\rm Ap}(S,a_1)$.
  Let $s\in H$ suppose that $s\notin {\rm Ap}(S,a_1)$, hence $s-a_1\in S$, by the above decomposition there exists   unique $s_i\in H, l\in\bn$ such that 
$s-a_1=s_i+la_1 $ that is $s =s_i+(l+1)a_1 $ a contradiction to the direct sum decomposition. Reciprocally, let $s\in {\rm Ap}(S,a_1)$, then there exists
 unique $s_i\in H, l\in\bn$ such that 
$s=s_i+la_1 $, if $l>0$ then $s-a_1\in S$ a contradiction , hence $s=s_i\in H$.
\item If $s=l_2a_2+\ldots +l_na_n $ with $(k_1,\ldots ,k_n)\not= (l_1,\ldots ,l_n)$ and $x_2^{l_2}\ldots x_n^{l_n} \prec _{revlex} x_2^{k_2}\ldots x_n^{k_n}$ then 
$  x_2^{k_2}\ldots x_n^{k_n}\in in(I(S))$, a contradiction.

\end{enumerate}
 \end{proof}
\begin{example}Let $n=2$, and  $a_1,a_2$ be natural numbers, $\lambda =gcd(a_1,a_2)$, we have that 
$K[t^{a_1},t^{a_2}]\simeq  K[x_1,x_2]/(x_2^ {a_1/\lambda }-x_1^{a_2/\lambda })$ it is clear that
that $x_2^ {a_1/\lambda }-x_1^{a_2/\lambda }$ is a \gbb of the ideal $(x_2^ {a_1/\lambda }-x_1^{a_2/\lambda })$ for  $\prec _{revlex}$.
 We have $in(x_2^ {a_1/\lambda }-x_1^{a_2/\lambda })=(x_2^ {a_1/\lambda })$, hence
$\displaystyle K[t^{a_1},t^{a_2}]= \simeq \oplus_{k=0}^{{a_1/\lambda }-1}K[t^{a_1}][t^{ka_2}]$, the 
Poincaré series is given by: $\displaystyle P_A(t)=\frac{\sum_{k=0}^{{a_1/\lambda }-1}t^{ka_2 }}{1-t^{a_1}} $. if $a_1,a_2$ are coprime then 
 the Frobenius number is $(a_1-1)a_2-a_1=(a_1-1)(a_2-1)-1.$
\end{example}
\section{Frobenius number, Hilbert-Poincaré series, the case $n=3$}
This section is a short version of  \cite{Mo1} and \cite{Mo2}.

Let consider three natural numbers $a,b,c$ and $S$ be the semigroup generated by $a,b,c$.\hfill\break 
First, remark that any solution $\alpha :=(u,v,w)$ of the Diophantine equation $ ua+vb+wc=0$
 gives rise to a binomial in the ideal $I(S)$ in the following way:\hfill\break  we write the vector $ \alpha=\alpha_+- \alpha_-$, where the components of both 
$ \alpha_+,\alpha_-$ are nonnegative then $\ul {\bf x}^{\alpha_+}- \ul {\bf x}^{\alpha_-}\in I(S)$, where 
$\ul {\bf x}^{\alpha_+}=x_1^{{\alpha_+}_1}x_2^{{\alpha_+}_2}  x_3^{{\alpha_+}_3} $. Reciprocally if 
$\ul {\bf x}^{\alpha}- \ul {\bf x}^{\beta }\in I(S)$  and $\ul {\bf x}^{\alpha},\ul {\bf x}^{\beta }$ have not common 
factors then $(u,v,w):=\alpha -\beta $ is a
solution of the equation $ ua+vb+wc=0$.\hfill\break 
Second, it is clear that find solutions  $(u,v,w)$  of the Diophantine equation $ ua+vb+wc=0$ is equivalent 
to find solutions $(s,p,r)$ of  the Diophantine equation $sb-pc=ra$.\hfill\break 
Let $s_0$ be the smallest natural number such that $(s_0,0,r_0)$ is solution of the equation $sb-pc=ra$. We set $p_0=0$.\hfill\break 
Let $p_1$ be the smallest natural number such that $(s_1,p_1,r_1)$ is solution of the equation $sb-pc=ra$, with $0\leq s_1<s_0$. Note that
 $s_0=\frac{a}{\gcd(a,b)}$ and
$p_1=\frac{\gcd(a,b)}{\gcd(a,b,c)}$. The numbers 
$s_1$ can be got from the extended Euclide's algorithm for the computation of the greatest common divisor of $a,b$. 

Let consider the Euclides' algorithm with negative rest: 
$$ \left \{
  \begin{array}{r c l}
      s_{0} & = & q_2s_1-s_2 \\
      s_1 & = & q_3s_2-s_3 \\
      \ldots  & = & \ldots \\
      s_{m-1} & = & q_{m+1}s_{m} \\
      s_{m+1} & = & 0 
   \end{array}
   \right .$$
  $ q_i\ge 2 \ ,\ s_i\ge 0\ \   \forall i$.
 
Let define the sequences: 
$ p_i, r_i$ $(0\leq i\leq m+1)~,$  
 by:
$$  p_{i+1}=p_iq_{i+1}-p_{i-1}~,  r_{i+1}=r_iq_{i+1}-r_{i-1}~, (1\leq i\leq m) .$$
Note that from \cite{Mo2} we have for $i=0,\ldots ,m$ that $s_i p_{i+1}- s_{i+1}p_i=s_0p_1=\frac{a}{\gcd(a,b,c)}$.
Let $ \mu$ 
 the unique integer such that   $ r_{\mu }> 0\geq r_{\mu+1}$.
 \begin{theorem} (\cite{Mo0}, \cite{Mo1} and \cite{Mo2})
 The set $$x_2^{s_\mu }-x_1^{r_\mu }x_3^{p_\mu }, x_3^{p_{\mu +1}}-x_1^{-r_{\mu +1}}x_2^{s_{\mu +1}},
 x_2^{s_\mu -s_{\mu +1}}x_3^{p_{\mu +1}- p_\mu }-x_1^{r_\mu -r_{\mu +1}}$$
 is a \gbb of $I(S)$ for $\prec _{revlex}$ with $x_3,x_2,x_1$.
 In particular $in(I(S)=(x_2^{s_\mu }, x_3^{p_{\mu +1}},x_2^{s_\mu -s_{\mu +1}}x_3^{p_{\mu +1}})$, and 
 $$\bn^2\setminus exp(in(I(S))=\{ (k,l)\in \bn^2 | 0\leq k<s_\mu -s_{\mu +1}, 0\leq l<p_{\mu +1} \}\cup $$
$$\{ (k,l)\in \bn^2 | s_\mu -s_{\mu +1}\leq k<s_\mu , 0\leq l<p_{\mu +1}- p_{\mu}\}.$$

$$K[t^{a},t^b,t^c]\simeq \oplus_{(k,l) \in \bn^2\setminus exp(in(I(S))} K[t^{a}][t^{kb+lc }].$$ 
 In particular the Poincaré series is given by: $$P_A(t)=\frac{\sum_{(k,l) \in \bn^2\setminus exp(in(I(S))}t^{kb+lc }}{1-t^{a}} $$
 and if the numbers $a,b,c$ are relatively prime the  Frobenius number is $$g(S)=\max { kb+lc |  (k,l)\in  \bn^2\setminus exp(in(I(S))}-{a_1}.$$
\begin{proof}we can give a new and shorter proof than the one given in the general case in \cite{Mo2}.
We can assume that the numbers $a,b,c$ are relatively prime. Let consider the three elements of $I(S)$: $x_2^{s_\mu }-x_1^{r_\mu }x_3^{p_\mu }, x_3^{p_{\mu +1}}-x_1^{-r_{\mu +1}}x_2^{s_{\mu +1}},
 x_2^{s_\mu -s_{\mu +1}}x_3^{p_{\mu +1}- p_\mu }-x_1^{r_\mu -r_{\mu +1}}$. 
It then follows that $J:=(x_2^{s_\mu }, x_3^{p_{\mu +1}},x_2^{s_\mu -s_{\mu +1}}x_3^{p_{\mu +1}})\subset in(I(S)$. 
Now we count the numbers of monomial not in $J$, $$\card (\bn^2\setminus exp(J)) =(s_\mu -s_{\mu +1})p_{\mu +1}+s_{\mu +1}(p_{\mu +1}-p_\mu )=
s_\mu p_{\mu +1}-s_{\mu +1}p_\mu =a.$$
On the other hand by Corollary \ref{ap1}, $\card (\bn^2\setminus exp(in(I(S)))=a$ this implies $in(I(S)=J$, hence the set 
$x_2^{s_\mu }-x_1^{r_\mu }x_3^{p_\mu }, x_3^{p_{\mu +1}}-x_1^{-r_{\mu +1}}x_2^{s_{\mu +1}},
 x_2^{s_\mu -s_{\mu +1}}x_3^{p_{\mu +1}- p_\mu }-x_1^{r_\mu -r_{\mu +1}}$
 is a \gbb of $I(S)$ for $\prec _{revlex}$ with $x_3,x_2,x_1$. The other claims follows from the Theorem \ref{noether}
 \end{proof}
\end{theorem}
\section{Algorithm for the case $n\geq 4$}
For $n=3$, we have seen that the  algorithm use only Euclide's algorithm. Let $n \geq 4$.
Let    $  a_1 ,\ldots ,  a_{n}  $ be relatively prime natural numbers,  $S$  the semigroup generated
by $a_1 ,\ldots ,  a_{n}$.
Let $R:=K[x_{1},\ldots,x_{n}]$ be the polynomial ring graded by the weights  $\deg x_1=a_1,\ldots ,\deg x_n=a_n$.  
We consider $\prec _{degrevlex}$  the degree reverse lexicographical order with $x_n,\ldots ,x_1$, $A=K[t^{a_1}, \ldots , t^{a_n}]\simeq R/I(S).$

\begin{remark} 
If $x_2^{k_2}\ldots x_n^{k_n}  \in    in(I(S)$ is part of a minimal generating system of $in(I(S)$, then all the monomials with exponents in the cube 
$[0; k_2]\times\ldots\times [0; k_n]$ are not in $ in(I(S)$ except $x_2^{k_2}\ldots x_n^{k_n}$.
that is we have $(k_2+1)\times\ldots\times (k_n+1)-1$ monomials not in $ in(I(S)$, hence $(k_2+1)\times\ldots\times (k_n+1)-1\leq a_1$ , which implies
$S_{n-1}+S_{n-2}+ \ldots+S_2+S_1\leq a_1$, where $S_i$ is the symmetric polynomial of degree $i$ in the variables ${k_2},\ldots ,{k_n} $.
 In particular ${k_2}+\ldots +{k_n}\leq  a_1$.
 \end{remark}
  We have the following algorithm for the Frobenius number:

 \begin{algorithm} Frobenius MM-DD:
 
 {\bf Input}: $a_1 ,\ldots ,  a_{n}$
 
 {\bf  Ouput}: the Apéry set of $S$ with respect to $a_1$. The Frobenius number of $S$.
 
{\bf Begin}
\begin{enumerate}
\item  $sum=1; $ test=false,testsum=false, $Ap=\{ 0\} $, $Apmod=\{ 0\} $.
\item while ($sum\leq a_1$) and (testsum=false) do
\begin{enumerate}

\item for each monomial   $x_2^{k_2}\ldots x_n^{k_n}$ with ${k_2}+\ldots+ {k_n}=sum$ 
\item if $\prod_{i=2}^n (k_i+1)\leq a_1+1$ then
\begin{enumerate}
\item compute  $qsum:={k_2}a_2+\ldots+ {k_n}a_n$ and the congruence class mod $a_1$, that is ${k_2}a_2+\ldots+ {k_n}a_n \mod a_1$
\item  If $qsum \mod a_1$ doesn't belongs to $Apmod$ then  $Ap=Ap\cup \{ qsum\} $,
$Apmod=Apmod\cup \{ qsum \mod a_1\} $.
 \item  If $qsum \mod a_1$  belongs to $Apmod$, let $h\in Ap $ such that $h=qsum \mod a_1$
\begin{enumerate}
\item test=true
\item  If  $qsum<h$ then $Ap=(Ap\setminus \{ h\}\cup \{ qsum\} $. 
 
\end{enumerate}
\end{enumerate}
 \item end if ($\prod_{i=2}^n (k_i+1)\leq a_1+1$)
\item If either $\card (Apmod)=a_1 $ or for all monomials  $x_2^{k_2}\ldots x_n^{k_n}$ with ${k_2}+\ldots+ {k_n}=sum$, we have that test=true
then testsum=true. 
\item Otherwise testsum=false. $sum=sum+1$.
\end{enumerate}
\item 
$frob+a_1=\max {p\in Ap }$.
\end{enumerate}
{\bf End} (of the algorithm).

\end{algorithm}
The knapsack integer problem: Let $b\in \bn$, in order to know if $b\in S$ we run the above algorithm, let $h\in Ap$ such that  $b=h\mod a_1$ ,then $b\in S$
 if and only if $b\geq h$.
%%%%%%%%%%%%Algorithme 2

\begin{algorithm} Frobenius NoINI MM-DD:
 
 {\bf Input}: $a_1 ,\ldots ,  a_{n}$
 
 {\bf  Ouput}: the Apéry set of $S$ with respect to $a_1$. The Frobenius number of $S$. The set of monomials
 in the variables $x_2,...,x_n$ not in the ideal $ in(I(S)$.
 $Sumnoini$ the least upper bound for the usual degree of a minimal system of monomial generators of the ideal $ in(I(S)$
 
{\bf Begin}
\begin{enumerate}
\item  $sum=1; $ test=false,testsum=false, $Ap=\{ 0\} $, $Apmod=\{ 0\}, Lnoini=\{ 1\}  $.
\item while ($sum\leq a_1$) and (testsum=false) do
\begin{enumerate}

\item for each monomial   $x_2^{k_2}\ldots x_n^{k_n}$ with ${k_2}+\ldots+ {k_n}=sum$ 
\item if $\prod_{i=2}^n (k_i+1)\leq a_1+1$ then
\begin{enumerate}
\item compute  $qsum:={k_2}a_2+\ldots+ {k_n}a_n$ and the congruence class mod $a_1$, that is ${k_2}a_2+\ldots+ {k_n}a_n \mod a_1$
\item  If $qsum \mod a_1$ doesn't belongs to $Apmod$ then  $Ap=Ap\cup \{ qsum\} $,
$Apmod=Apmod\cup \{ qsum \mod a_1\} $, $Lnoini=Lnoini=\cup \{ x_2^{k_2}\ldots x_n^{k_n}\}$.
 \item  If $qsum \mod a_1$  belongs to $Apmod$, let $h\in Ap $ such that $h=qsum \mod a_1$, and 
 $x_2^{j_2}\ldots x_n^{j_n}\in Lnoini$ which $qsum$ is $h$.
\begin{enumerate}
\item test=true
\item  If  $x_2^{k_2}\ldots x_n^{k_n}\prec _{revlex} x_2^{j_2}\ldots x_n^{j_n}$ then $Ap=(Ap\setminus \{ h\}\cup \{ qsum\} $, and
$Lnoini=(Lnoini\setminus \{ x_2^{j_2}\ldots x_n^{j_n}\}\cup \{ x_2^{k_2}\ldots x_n^{k_n}\} $. 
 
\end{enumerate}
\end{enumerate}
 \item end if ($\prod_{i=2}^n (k_i+1)\leq a_1+1$)
\item If either $\card (Apmo)=a_1 $ or for all monomials  $x_2^{k_2}\ldots x_n^{k_n}$ with ${k_2}+\ldots+ {k_n}=sum$, we have that test=true
then testsum=true. 
\item Otherwise testsum=false. $sum=sum+1$.
\end{enumerate}
\item 
$frob+a_1=\max {p\in Ap }$.
\end{enumerate}
{\bf End} (of the algorithm).

\end{algorithm}

%%%%%%%%%%%%Algorithme 3

\begin{algorithm} Frobenius INI,GB MM-DD:
 
 This algorithm needs to run previuosly the algorithme Frobenius NOINI MM-DD.
 
 {\bf Input}: $a_1 ,\ldots ,  a_{n}$, $Ap, Apmod,Lnoini  $.
 
 {\bf  Ouput}: the Apéry set of $S$ with respect to $a_1$. The Frobenius number of $S$. The set of monomials
 in the variables $x_2,...,x_n$ not in the ideal $ in(I(S)$.
 
{\bf Begin}
\begin{enumerate}
\item  $sum=1; $ test=false,testsum=false, $ Lini=\emptyset   $.
\item while ($sum\leq {\text Sumnoini}$) and (testsum=false) do
\begin{enumerate}

\item for each monomial   $x_2^{k_2}\ldots x_n^{k_n}$ with ${k_2}+\ldots+ {k_n}=sum$ 
\item if $x_2^{k_2}\ldots x_n^{k_n}$ belongs to the ideal generated by $Lini$ then $testini=true.$
\item If $testini=false$ then
\item if $\prod_{i=2}^n (k_i+1)\leq a_1+1$ then

\item compute  $qsum(\underline {\bf k}):={k_2}a_2+\ldots+ {k_n}a_n$ and the congruence class mod $a_1$, that is ${k_2}a_2+\ldots+ {k_n}a_n \mod a_1$
\begin{enumerate}
\item Test if $x_2^{k_2}\ldots x_n^{k_n}$ is a generator of$ in(I)$:\hfill\break 
 let $x_2^{j_2}\ldots x_n^{j_n}\in Lnoini$ such that its $qsum(\underline {\bf k})=qsum(\underline {\bf j}) \mod a_1$.
 If $\underline {\bf j}\not= \underline {\bf k}$ and $\underline {\bf j}\not, \underline {\bf k}$ have disjoint supports 
 then $Lini=Lini\cup \{ x_2^{k_2}\ldots x_n^{k_n}\} $

\item test=true
 
\end{enumerate}
 \item end if ($\prod_{i=2}^n (k_i+1)\leq a_1+1$)
  \item end if ($testini=false$)
\item If either $\card (Apmo)=a_1 $ or for all monomials  $x_2^{k_2}\ldots x_n^{k_n}$ with ${k_2}+\ldots+ {k_n}=sum$, we have that test=true
then testsum=true. 
\item Otherwise testsum=false. $sum=sum+1$.
\end{enumerate}
\item 
$frob+a_1=\max {p\in Ap }$.
\end{enumerate}
{\bf End} (of the algorithm).

\end{algorithm}

In the following example we compute the \gbb  by my software and by Cocoa, it runs in above 7 seconds in both softwares. The \gbb has 571 generators.
 
\begin{example}$(a_1=1030 ,  a_2= 1031, a_3= , 1034,a_4=1039,a_5=1046,a_6=1055  ,a_7=1066,a_8=1079  ,a_9=1094  ,a_{10}=1111  ,a_{11}=1130 
 ,a_{12}=1151  ,a_{13}=1373  
,a_{14}=1393  ,a_{15}=1423, a_{16}=1433, a_{17}=1493 )$.

If we compute only the Apéry number our algorithm is much faster than Cocoa. In this example the Apéry number is   5145. 
\end{example}


\begin{thebibliography}{GHS}
\bibitem{ap}R. Apéry, Sur les branches superlinéaires des courbes algébriques, C.R. Acad. Sci. Paris 222 (1946), 1198-1200.
\bibitem{Ei}  D. Eisenbud, Commutative algebra with a view toward algebraic geometry. 
Graduate Texts in Math., vol. 150 (1995), Springer-Verlag, Berlin and New York.
\bibitem{ein} Einstein, David, et al. "Frobenius numbers by lattice point enumeration.." Integers 7.1 (2007)
\bibitem{LJ}Monique Lejeune-Jalabert, Effectivité des calculs polynomiaux, Courd de DEA 1984-1985, Institut Fourier, Université de Grenoble I.
\bibitem{Mo0} Morales Marcel, {\sl Fonctions de Hilbert, genre géométrique
d'une
singularité quasi-homogène Cohen-Macaulay}. CRAS Paris, t.301,
série A $n^o$ 14 {\bf (1985)}.
\bibitem{Mo1} Morales M.- Syzygies of monomial curves and a linear diophantine problem of Frobenius, Preprint Max Planck Institut für Mathematik (Bonn-RFA) (1987

\bibitem{Mo2}
{M. Morales,
\'Equations des vari\'et\'es monomiales en codimension deux,
J. Algebra \textbf{175} (1995) 1082-1095.}
\bibitem{rod}J. Rodseth, On a linear Diophantine problem of Frobenius, J. reine angew. Math., 301(1978) 171-178
\bibitem{rou} B.H. Roune, Solving Thousand Digit Frobenius Problems Using Grobner Bases, Journal of Symbolic Computation volume 43(1)  2008,1-7. 
\bibitem{ra}J.L. Ramírez Alfonsín, The Diophantine Frobenius Problem,
Oxford Lectures Series in Mathematics and its Applications 30,
Oxford University Press, (2005), 256 pages. 
\end{thebibliography}
\end{document}